\documentclass{amsart}

\usepackage{geometry} 
\usepackage{amssymb}
\usepackage{relsize}
\usepackage[all]{xy}
\usepackage{xcolor}
\usepackage{dsfont}
\usepackage[hidelinks]{hyperref}
\usepackage{cleveref}

\numberwithin{equation}{section}

\newtheorem{theorem}[subsubsection]{Theorem}
\newtheorem{proposition}[subsubsection]{Proposition}

\newtheorem{corollary}[subsubsection]{Corollary}
\newtheorem{lemma}[subsubsection]{Lemma}

\newtheorem*{theorem*}{Theorem}

\theoremstyle{definition}
\newtheorem{definition}[subsubsection]{Definition}

\theoremstyle{remark}
\newtheorem{example}[subsubsection]{Example}
\newtheorem{remark}[subsubsection]{Remark}

\title[]{On geometrically reductive tensor categories}
\author{Kevin Coulembier}
\address{School of Mathematics and Statistics, University of Sydney, Australia}
\email{kevin.coulembier@sydney.edu.au}

\newcommand{\Sym}{\operatorname{S}}
\newcommand{\w}{\mathlarger{\wedge}}
\renewcommand{\k}{\Bbbk}
\newcommand{\bk}{\Bbbk}

\newcommand{\unit}{\mathds{1}}
\renewcommand{\Vec}{\mathtt{Vec}}
\newcommand{\sVec}{\mathtt{sVec}}
\newcommand{\Ver}{\mathtt{Ver}}
\newcommand{\Rep}{\mathtt{Rep}}
\newcommand{\Ind}{\mathtt{Ind}}
\newcommand{\VEC}{\mathtt{VEC}}
\newcommand{\Tilt}{\mathtt{Tilt}}

\newcommand{\GL}{\operatorname{GL}}
\newcommand{\Indu}{\operatorname{Ind}}
\newcommand{\Res}{\operatorname{Res}}
\newcommand{\SL}{\operatorname{SL}}
\newcommand{\ch}{\operatorname{ch}}
\newcommand{\CAlg}{\operatorname{CAlg}}
\newcommand{\Dim}{\mathrm{Dim}}

\newcommand{\inv}{\mathrm{inv}}

\newcommand{\mZ}{\mathbb{Z}}
\newcommand{\mN}{\mathbb{N}}
\newcommand{\mR}{\mathbb{R}}
\newcommand{\mG}{\mathbb{G}}

\newcommand{\bI}{\mathbf{I}}

\newcommand{\cC}{\mathcal{C}}
\newcommand{\cD}{\mathcal{D}}
\newcommand{\cT}{\mathcal{T}}
\newcommand{\cTb}{\overline{\cT}}

\newcommand{\ev}{\mathrm{ev}}

\newcommand{\PS}{\operatorname{PS}}
\newcommand{\PE}{\operatorname{PE}}
\newcommand{\Hom}{\operatorname{Hom}}
\newcommand{\End}{\operatorname{End}}
\newcommand{\mV}{\mathbb{V}}

\newcommand{\tto}{\twoheadrightarrow}
\begin{document}

\begin{abstract}
We prove the conjecture that higher Verlinde categories are geometrically reductive. This is one of the two properties required in order for recent results on algebraic geometry in tensor categories to apply to these categories. We also reduce two further conjectures concerning geometric reductivity for tensor categories to other conjectures appearing in the literature.
\end{abstract}



\maketitle

\section*{Introduction}

Symmetric tensor categories of moderate growth over fields of characteristic zero are well-understood due to classical work of Deligne, see \cite{Del90, Del02}. Throughout, we let $\Bbbk$ be an algebraically closed field, of prime characteristic $p>0$.
The current paper is part of the ongoing exploration of symmetric tensor categories of moderate growth over $\bk$. By \cite{CEO2, Del90}, every such tensor category is a representation category of an affine group scheme internal to an `incompressible' tensor category, which in a sense reduces the study to the latter type of category. There is a concrete conjecture for what these incompressible categories might be, namely the tensor subcategories of a category $\Ver_{p^\infty}$ introduced in \cite{BE, BEO, AbEnv}, and studied further and generalised in \cite{Dec25, Ne}. Evidence for this conjecture can be found in \cite{CEO, EG} and references therein. 

To exploit truly the realisation of a tensor category as a representation category in an incompressible category, results in \cite{CAlg, Geom, CEO2, CS, CS2} (concerned with algebraic geometry, existence of homogeneous spaces and characterisations of simple algebras) demonstrate that the incompressible categories should satisfy two fundamental properties introduced in \cite[\S 3]{CAlg}, they should be `maximally nilpotent' (MN) and `geometrically reductive' (GR). The former means that symmetric algebras of objects are as small as possible, while the latter is a direct generalisation of the property with same name for affine group schemes. It is known by \cite{CEO2} that $\Ver_{2^\infty}$ satisfies these conditions and the same is conjectured of $\Ver_{p^\infty}$ for all $p$ in \cite{CAlg}. That $\Ver_{p^2}\subset \Ver_{p^\infty}$ is GR was proved in \cite[Proposition~3.1.8]{CAlg}, with a method that unfortunately already fails for $\Ver_{p^3}$.

The two conditions MN and GR have rather different roles. The MN condition is extremely restrictive, in fact expected to completely characterise incompressible categories. The GR condition is expected to hold much more generally, although proving this is not always easy. For instance, the classical conjecture of Mumford \cite{Mu} that reductive groups are geometrically reductive was proved by Haboush in \cite{Ha}.

In conclusion, the main aims of the line of research described above are showing that $\Ver_{p^\infty}$ contains all incompressible tensor categories and is MN and GR. The current paper proves this last condition, confirming \cite[Conjecture~3.1.6(2)]{CAlg}:

\medskip

{\bf Theorem A.} {\em For any $p>0$, the tensor category $\Ver_{p^\infty}$ is geometrically reductive.}

\medskip

For $p=2$, Theorem~A was proved as \cite[Theorem~9.3.7]{CEO} using different methods, relying on the alternative construction of $\Ver_{2^\infty}$ in \cite{BE}. We therefore focus on a proof for $p>2$ specifically, which allows for some simplifications, culminating in Theorem~\ref{thm:GR}. 
The approach to the proof is as follows. We write an important instance of the geometric reductivity of $\SL_2$ as a commutative diagram in $\Tilt \SL_2$, and subsequently demonstrate that the image of this diagram under the defining functor $\Tilt \SL_2\to\Ver_{p^n}$ implies the target is also GR.
In doing the latter we actually obtain some further results towards proving that $\Ver_{p^\infty}$ is also MN, see Remark~\ref{rem:MN}.

As mentioned above, the GR condition is not expected to be very exclusive. We rather view Theorem~A as a step towards proving that all finite tensor categories are GR. In fact, we show that this is the case, provided some of the conjectures mentioned above are valid:

\medskip

{\bf Theorem B.} {\em Assume that $\Ver_{p^\infty}$ is maximally nilpotent, then any finite tensor category that admits a tensor functor to $\Ver_{p^\infty}$ is geometrically reductive.}

\medskip

Theorem~B is proved in Theorem~\ref{thm:comb}. The same theorem also proves that the analogue of Nagata's theorem \cite{Na} (that an affine group scheme $G$ is GR if and only if the invariant algebras of finitely generated algebras with $G$-actions are finitely generated) for tensor categories follows from the above conjectures.

The paper is organised as follows. In Section~\ref{sec:prel} we review some background. In Section~\ref{sec:repth} we derive some results on modular representation theory that are needed for the sequel. In Section~\ref{sec:main} we prove Theorem~A and in Section~\ref{sec:inv} we prove Theorem~B.

\section{preliminaries}\label{sec:prel}

Denote by $S_n$ the permutation group of the set $\{1,2,\ldots,n\}$, and by $|\sigma|$ the parity of a permutation $\sigma\in S_n$. We set $\mN=\mZ_{\ge 0}$ and 
let $\bk$ be an algebraically closed field.

\subsection{Tensor categories}

\subsubsection{} We use the notion of `tensor categories' as in \cite{EGNO}, except that, since the topic is fundamentally restricted to \emph{symmetric} tensor categories, we leave out any reference to `symmetric'. Concretely, an essentially small $\bk$-linear symmetric category
$(\cC,\otimes,\unit)$ is a {\bf tensor category over $\bk$} if $\cC$ is abelian with the length $\ell(X)$ of each object $X\in\cC$ finite, $\bk\to\End_{\cC}(\unit)$ is an isomorphism and the monoidal category $(\cC,\otimes,\unit)$ is rigid. The latter means that every $X\in\cC$ has a monoidal dual $X^\ast$, with appropriate evaluation morphism $\ev_X:X^\ast\otimes X\to \unit$ and coevaluation. An exact symmetric monoidal functor between tensor categories is called a {\bf tensor functor}.

A tensor category is {\bf of moderate growth} if for each $X\in \cC$ there is $C\in\mR$ with $\ell(X^{\otimes n})\le C^n$. Examples of such categories are given by {\bf finite tensor categories} $\cC$ in the sense of \cite{EGNO}, meaning $\cC$ has enough projective objects and only finitely many simple objects up to isomorphism. A tensor category $\cC$ is {\bf finitely generated} if there exists an $X\in\cC$ such that every object in $\cC$ is a subquotient of a polynomial in $X$ and $X^\ast$.

The smallest tensor category over $\bk$ is the category $\Vec$ of finite dimensional vector spaces. We denote the symmetric monoidal category of all vector spaces by $\VEC$. The tensor category $\sVec$ of
supervector spaces is the category of $\mZ/2$-graded vector spaces with braiding such that the braiding isomorphism on $\bar{\unit}\otimes\bar{\unit}$ for the non-trivial simple object $\bar{\unit}$ is $-1.$ For an affine group scheme $G$ over $\bk$, the category of finite dimensional rational representations $\Rep_{\bk}G$ is canonically a tensor category.

For any ind-object $X$ in a tensor category $\cC$, we consider the `invariant' subobject $X^{\inv}=\Hom(\unit, X)$. We can regard $-^{\inv}$ as the right adjoint of the inclusion tensor functor $\Vec\to\cC$.

We denote by $\CAlg\cC$ the category of commutative ind-algebras in $\cC$ (monoid objects in $\Ind\cC$). For $A\in \CAlg$, we have that $A^{\inv}\subset A$ is its maximal subalgebra belonging to $\VEC\subset\Ind\cC$. An ind-algebra is finitely generated if it is a quotient of an algebra $\Sym X$, as defined next, for $X\in\cC$.

\subsubsection{Some presentations}

For $X\in\cC$ consider the braiding morphism $s=s_{X,X}:X^{\otimes 2}\to X^{\otimes 2}$.  
Assume $\mathrm{char}(\k)\not=2$, then we have projection operators $(1\pm s)/2$, leading to the decomposition $X^{\otimes 2}=\Sym^2X\oplus\w^2X$ into the symmetric and exterior power. More generally, we let $\Sym^jX$ be the maximal $S_j$-invariant (for the braiding action) quotient of $X^{\otimes j}$, so that $\Sym X:=\oplus_j\Sym^jX$ is the commutative algebra freely generated by $X$. This means we have a defining exact sequence
\begin{equation}\label{eq:PS}
\PS^j(X)\to X^{\otimes j}\to \Sym^jX\to0,
\end{equation}
where
\[\PS^j(X)=\bigoplus_{i=0}^{j-2}\PS^j_{i}(X),\quad\mbox{where}\quad \PS_i^j(X)= X^{\otimes i}\otimes\w^2 X\otimes X^{\otimes(j-i-2)}, \]
and the left map in \eqref{eq:PS} simply consists of inclusions of direct summands.

Similarly, we let $\w^jX$ be the maximal $S_j$-invariant, for the sign twisted braiding action, quotient of $X^{\otimes j}$. This means we have a defining exact sequence
\begin{equation}\label{eq:PE}
\PE^j(X)\to X^{\otimes j}\to \w^jX\to0,
\end{equation}
where
\[\PE^j(X)=\bigoplus_{i=0}^{j-2}\PE^j_{i}(X),\quad\mbox{where}\quad \PE_i^j(X)= X^{\otimes i}\otimes\Sym^2 X\otimes X^{\otimes(j-i-2)}, \]
and the left map in \eqref{eq:PE} again consists of inclusions of direct summands.
We will also use the presentation
\begin{equation}
\label{eq:PE2}
\Sym^2X\otimes \w^{j-2}X \to  X\otimes \w^{j-1}X\to\w^jX\to 0,
\end{equation}
where the left map is the composition of the obvious inclusion into $X^{\otimes 2}\otimes \w^{j-2}X$ followed by the obvious projection. This sequence is exact, for instance by observing that there is a map from the sequence in \eqref{eq:PE}, which is the identity on the right term $\w^j X$, and an epimorphism for the other two terms.

\subsubsection{Geometric reductivity}
Following \cite{CAlg}, a tensor category $\cC$ is {\bf geometrically reductive} (GR) if for every non-zero morphism $X\to\unit$, there is some $m\in\mZ_{>0}$ for which the induced epimorphism $\Sym^m X\tto \unit$ is split. For an affine group scheme $G$ over $\bk$, the tensor category $\Rep_{\bk}G$ is GR if and only if $G$ is geometrically reductive in the classical sense.
We recall the following connection with `invariant theory'.
\begin{lemma}\cite[Theorem~3.1.5(1)]{CAlg}\label{lem:315} If for every finitely generated $A\in\CAlg\cC$, the invariant algebra $A^{\inv}$ is finitely generated, then $\cC$ is GR.
\end{lemma}
For $\cC=\Rep_{\bk}G$ this is actually `if and only if', see \cite{Na}, and we prove a more general result for tensor categories in that direction in the current paper, see Theorem~\ref{thm:comb}(2).

\subsection{Higher Verlinde categories}
Let $\bk$ be of characteristic $p>0$.
\subsubsection{Representations of $\SL_2$}

Consider the algebraic group $\SL_2$ in its standard realisation. We label the simple modules, by their highest weight, with respect to the Borel subgroup of upper-triangular matrices as $L_i,i\in\mN$, and similarly for the indecomposable tilting modules $T_i$, see \cite{Jantzen}. A special role is played by the Steinberg modules $St_i:=L_{p^i-1}=T_{p^i-1}$. The non-zero tensor ideals of the monoidal category of tilting modules form one chain
$$\Tilt \SL_2\;\supset\; \bI_1\;\supset\; \bI_2\;\supset\; \bI_3\;\supset\;\cdots,$$
see \cite{Selecta} and are characterised by $St_i\in \bI_i$ and $St_{i-1}\not\in\bI_{i}$.

\subsubsection{}The higher Verlinde category $\Ver_{p^n}$, for $n\in\mZ_{>0}$, is the symmetric tensor category constructed in \cite{BEO,AbEnv} as the abelian envelope of $\Tilt SL_2/\bI_n$. It thus has a defining symmetric monoidal functor from $\Tilt SL_2$. Moreover, $\Ver_{p^n}$ is a finite tensor category and the objects in~$\bI_{n-1}$ are sent to the projective objects. Following \cite{Ne}, we can extend this functor to a symmetric monoidal functor
\[F_n:\cTb_n\to\Ver_{p^n}\]
where $\cTb_n\supset \Tilt SL_2$ is the full monoidal subcategory of $\Rep \SL_2$ of modules $W$ for which $W\otimes St_{n-1}\in\Tilt SL_2$. An advantage over the traditional defining functor is that every simple object in $\Ver_{p^n}$ is in the image of this functor. Indeed, by \cite[Theorem~5.8]{Ne} the simple objects in $\Ver_{p^n}$ are given by $V_i:=F_n(L_i)$, $0\le i\le p^{n-1}(p-1)-1$. In particular, we set \[\mV:=V_{p^{n-1}-1}=F_n(St_{n-1})\;\in\Ver_{p^n},\]
which is thus a simple projective object.

We have inclusions of tensor subcategories $\Ver_{p^{n-1}}\subset\Ver_{p^n}$, see \cite[Theorem~1.3]{BEO}, sending $V_i$ to $V_{pi}$, and it will be convenient to refer to the union $\Ver_{p^\infty}:=\cup_n\Ver_{p^n}$.

\begin{remark}
One can verify directly that the above labelling of simple objects in $\Ver_{p^n}$ by $0\le i\le p^{n-1}(p-1)-1$ agrees with the labelling in \cite{BEO} by the same set, for instance by using that the top of the tilting module $T_{ap^r-1+b}$, for $1\le a\le p-1$ and $0\le b\le p^{r}-1$ is given by $L_{ap^r-1-b}$ and \cite[Theorem~4.42]{BEO}.
\end{remark}

\begin{lemma}\label{lem:splitexact}
If an exact (when considered in $\Rep\SL_2$) sequence in $\cTb_n$
\[\Sigma:\;X_2\to X_1\to X_0\to 0\]
is split after applying $St_{n-1}\otimes-$, then the sequence $F_n(\Sigma)$ is exact in $\Ver_{p^n}$.
\end{lemma}
\begin{proof}
Since $F_n$ is monoidal and additive, $F_n(St_{n-1}\otimes \Sigma)\cong \mV\otimes F_n(\Sigma)$ is split exact. By faithful exactness of $\mV\otimes-$, it follows that $F_n(\Sigma)$ is also exact.
\end{proof}

\subsubsection{}Assume $p>2$. If for some object $W\in \cTb_n$, a symmetric or exterior power $\Sym^jW$ or $\w^jW$ also belongs to $\cTb_n$, then the defining presentations \eqref{eq:PS} or \eqref{eq:PE} (both in $\Rep \SL_2$ and in $\Ver_{p^n}$) yield
canonical morphisms 
\[\Sym^j(F_n(W))\to F_n(\Sym^j W),\quad \mbox{or}\quad \w^j(F_n(W))\to F_n(\w^j W).\]
An important intermediate result in the current paper will be about establishing cases where such morphisms are isomorphisms. This is certainly not generally the case. Indeed, for $W=V_{p-2}=\bar{\unit}$ in $\sVec\subset\Ver_p$, we have $\w^j V_{p-2}\not=0$ for all $j$, while $\w^j L_{p-2}=0$ for $j\ge p$.


\section{Some representation-theoretic results}\label{sec:repth}
Let $\Bbbk$ be an algebraically closed field of characteristic $p>2$.
\subsection{General linear group}
We consider an $m$-dimensional $\k$-vector space $V$ with corresponding general linear group $\GL_m=\GL(V)$. We closely follow the conventions from \cite{Jantzen}. In particular, by picking an ordered basis of $V$ we consider subgroups $T<B<\GL(V)$ where $T$ is the maximal torus of diagonal matrices and $B$ the (negative) Borel subgroup of matrices with 0 above the diagonal. The weight lattice is $X=\Hom(T,\mG_m)$ and $2\rho\in X$ is the sum of positive roots. Simple $\GL(V)$-representations are labelled by their highest weight, with respect to the opposite Borel $B^+$ of $B$, as $L(\lambda)$, for dominant $\lambda\in X$.

The following type of result is standard, see for instance \cite{Stein}, but we provide a proof as we are mainly interested in the case $p<m$, often excluded.

\begin{lemma}\label{lem:GLsplit}
For each $j\in\mZ_{>1}$, the exact sequence \eqref{eq:PE2}
\[\Sym^2V\otimes \w^{j-2} V\to V\otimes \w^{j-1}V\to \mathlarger{\wedge}^j V\to 0\]
becomes split in $\Rep \operatorname{GL}(V)$ after applying $L((p-1)\rho)\otimes-$.
\end{lemma}
\begin{proof}
We claim that all simple constituents of all terms in the sequence become tilting modules after tensoring with the Steinberg module $St:=L((p-1)\rho)$ of $G:=\GL(V)$, so that the result follows from the extension-vanishing in \cite[Corollary~E.2]{Jantzen}.

To prove this claim, we can use \cite[3.19(4)]{Jantzen}:
\[St\;=\;\Indu^G_B(\k_{(p-1)\rho}),\quad\mbox{and hence}\quad St\otimes W\;\cong\; \Indu^G_B\left(\k_{(p-1)\rho}\otimes \Res^G_B W\right),\]
for any $W\in \Rep G$. For any simple constituent $W$ of a term of the exact sequence, all weight occurring with non-zero multiplicity in $\k_{(p-1)\rho}\otimes W$ are dominant (or, when $p=3$, at least $\mu+\rho$ is dominant for each such weight $\mu$). It then follows from Kempf's vanishing theorem \cite[Proposition~4.5]{Jantzen} (where we add in \cite[Proposition~5.4(a)]{Jantzen} in case $p=3$) that $St\otimes W$ has a good filtration. By self-duality it is a tilting module.
\end{proof}

In the diagram in Lemma~\ref{lem:diagram1} below, the first row is given by the presentation \eqref{eq:PS} for $\Sym^m (V^\ast\otimes V)$. The second row is the tensor product of the presentations \eqref{eq:PE} for $\w^m(V^\ast)$ and $\w^m V$, by making the identification 
\[\w^mV^\ast\otimes \w^m V\cong\;\unit,\quad \alpha_1\wedge \cdots\wedge \alpha_m \otimes v_1\wedge\cdots\wedge v_m\mapsto \sum_{\sigma\in S_m}(-1)^{|\sigma|}\alpha_1(v_{\sigma(1)})\cdots \alpha_m(v_{\sigma(m)}).\]
\begin{lemma}\label{lem:diagram1}
The solid diagram diagram described above can be completed to a commutative diagram in $\Rep \GL(V)$ of the form
\[\xymatrix{
\PS^m(V^\ast\otimes V)\ar[r]&(V^\ast\otimes V)^{\otimes m}\ar[dr]^{\ev^{\otimes m}}&&\\
\PE^m(V^\ast)\otimes V^{\otimes m}\oplus (V^\ast)^{\otimes m}\otimes \PE^m(V)\ar[r]\ar@{-->}[u]&(V^\ast)^{\otimes m}\otimes V^{\otimes m}\ar[r]\ar@{-->}[u]&\unit\ar[r]&0.
}\]
\end{lemma}
\begin{proof}
The right upward arrow can be chosen to be the morphism
\[ \alpha_1\otimes  \cdots\otimes \alpha_m \otimes v_1\otimes\cdots\otimes v_m\mapsto \sum_{\sigma\in S_m}(-1)^{|\sigma|}\alpha_1\otimes v_{\sigma(1)}\otimes \cdots \otimes\alpha_m\otimes v_{\sigma(m)},\]
clearly leading to a commutative triangle.

For the left arrow, we can choose the morphism to be zero on $(V^\ast)^{\otimes m}\otimes \PE^m(V)$. For the $\PE^m_i(V^\ast)\otimes V^{\otimes m}$ we deal with $i=0$, purely for notational convenience. The morphism 
\[\PE^m_0(V^\ast)\otimes V^{\otimes m}=\Sym^2(V^\ast)\otimes (V^\ast)^{\otimes m-2} \otimes V^{\otimes m}\to (V^\ast\otimes V)^{\otimes m},\] via the lower path, takes values in the subrepresentation spanned by vectors of the form
\[\sum_{\sigma\in S_m}(-1)^{|\sigma|}\alpha\otimes v_{\sigma(1)}\otimes \alpha\otimes v_{\sigma(2)}\otimes \alpha_3\otimes v_{\sigma(3)}\otimes \cdots \otimes\alpha_m\otimes v_{\sigma(m)}.\]
These all belong to $\w^2(V^\ast\otimes V)\otimes (V^\ast\otimes V)^{\otimes m-2}$ and hence the morphism factors indeed through $\PS_0^m(V^\ast\otimes V)\to (V^\ast\otimes V)^{\otimes m}$.
\end{proof}

\subsection{A consequence for $\SL_2$}

\begin{proposition}\label{prop:splitpres}
Let $L$ be a simple tilting module for $\SL_2$, then the exact sequence \eqref{eq:PE2}
\[\Sym^2L\otimes \w^{i-2} L\to L\otimes \w^{i-1}L\to \mathlarger{\wedge}^i L\to 0\]
is split in $\Rep \SL_2$ after applying $St_{j}\otimes-$ for all $i\in\mZ_{>1}$ if $m:=\dim L\le p^j$.
\end{proposition}
\begin{proof}
The $\SL_2$-representation $L$ defines a group homomorphism $\phi:\SL_2\to \GL(L)$. The (ordered) weight basis of $L=L_{m-1}$ over $\SL_2$ leads to a choice of maximal torus in $\GL(L)$ such that the corresponding pullback on weights is
\begin{equation}\label{eq:resweight}
X\to\mZ,\quad \alpha\mapsto 2,
\end{equation}
for every simple positive root $\alpha$.

By Lemma~\ref{lem:GLsplit}, the short exact sequence splits after taking the tensor product with $\Res L((p-1)\rho)$, where $\Res$ stands for the restriction along $\phi$. It is well-known that $L((p-1)\rho)$, as an $\SL(L)$-representation, is a direct summand of a tensor power of $L$ (this follows for instance from Schur-Weyl duality and projectivity of the corresponding simple module of the symmetric group). In particular, $\Res L((p-1)\rho)$ is in $\Tilt \SL_2$.

To determine in which cell this tilting module lives, we apply \cite[Lemma~2.11]{Ne}, stating that the cell can be determined by evaluating the character of $\Res L((p-1)\rho)$ on roots of 1 in $\k$. Moreover, \eqref{eq:resweight} implies that the computation of the character is identical to the case studied in \cite[Lemma~2.12]{Ne} (for a principal $SL_2$ in a reductive group in characteristic above the Coxeter number). In particular, 
\[\mZ[x,x^{-1}]\,\ni\,\ch \Res L((p-1)\rho)\;=\;\frac{g(x^p)}{g(x)}\quad\mbox{with}\quad g(x)=\prod_{i=1}^{m-1} (x^i-x^{-i})^{a_i}, \]
with $a_i$ the number of positive roots of $\GL(T)$ that are the sum of $i$ simple positive roots. If $\omega_{p^s}$ denotes a primitive $p^s$-th root of $1$, then clearly $g(\omega_{p^s})=0$ if and only if $p^s\le m-1$. Hence the character of $\Res L((p-1)\rho)$ does not evaluate to zero on $\omega_{p^{j+1}}$, meaning $\Res L((p-1)\rho)\not\in \bI_{j+1}$. Consequently, it follows that $St_j$ is a direct summand of a tensor product of $\Res L((p-1)\rho)$ (with another tilting module), so also $St_j$ splits the sequence.
\end{proof}

\begin{remark}\label{rem:staySL2}
Proposition~\ref{prop:splitpres} is perfect for our main application. However, there are variations covering different cases. For instances, by working entirely within $\Rep \SL_2$, and using \cite[Lemma~3.3]{BEO} or equivalently \cite[Lemma~4.3.4]{AbEnv}, it follows that
the exact sequence \eqref{eq:PE2}
\[\Sym^2T_s\otimes \w^{i-2} T_s\to T_s\otimes \w^{i-1}T_s\to \mathlarger{\wedge}^i T_s\to 0\]
is split in $\Rep \SL_2$ after applying $St_{j}\otimes-$ whenever
\[2s+(i-2)(s+3-i)\;<\; p^{j+1}.\]
In other words, for a fixed $s$, the sequences (for all $i$) are split by $St_j$ if
\[s^2/4+5s/2+1/4\;<\; p^{j+1}.\]

The value of Proposition~\ref{prop:splitpres} lies in the fact that the sequences for $L=T_{p^a-1}=St_a$ can be spilt by $St_a$ for all $a$, whereas the argument within this remark only predicts splitting by $St_{2a-1}$. 
Conversely, the argument within the current remark includes the observation that the sequences for $L=T_1$ are already split themselves, whereas Proposition~\ref{prop:splitpres} requires tensoring with $St_1$.
\end{remark}


\section{Main result}\label{sec:main}

In this section we assume that $p=\mathrm{char}(\bk)>2$.

\subsection{Exterior powers}We fix an $n\in\mZ_{>1}$.

\begin{theorem}\label{thm:w}
Let $L$ be a simple $\SL_2$-tilting module of dimension $m\le p^{n-1}$, so $L=T_{m-1}=L_{m-1}$. Then $\w^i L\in \cTb_n$ for all $i\in\mN$, and moreover the canonical morphism in $\Ver_{p^n}$
\[\w^i F_n(L)\;\to\; F_n(\w^i L)\]
is an isomorphism.
In particular $\w^{m}F_n(L)=\unit$ and $\w^iF_n(L)=0$ for $i>m$.
\end{theorem}
\begin{proof}
We use induction on $i$. The cases $i\le 1$ (even $i<p$) are obviously true. Assume that the claim is true for $i\le d$. In particular, when we consider the exact sequence \eqref{eq:PE2} for $j=d+1$ and $X=L$, the first two terms are by assumption in $\cTb_n$. By Proposition~\ref{prop:splitpres}, the exact sequence is split after applying $St_{n-1}\otimes-$, so also $\w^{d+1}L$ is in $\cTb_n$.
Moreover, Lemma~\ref{lem:splitexact} thus shows that that application of $F_n$ on the presentation \eqref{eq:PE2} of $\w^{d+1}L$ produces an exact sequence which can be identified with a defining presentation of $\w^{d+1}F_n(L)$ in $\Ver_{p^n}$, concluding the proof.
\end{proof}

\begin{example}\label{ex:w}
\begin{enumerate}
\item For $0\le j\le n-2$ and $1\le i< p$, in $\Ver_{p^n}$ we have
\[\w^{ip^j}V_{ip^j-1}=\unit,\quad\mbox{and}\quad \w^{s}V_{ip^j-1}=0\quad\mbox{for }s>ip^j.\]
Theorem~\ref{thm:w} also shows that one of the $p$-adic dimensions introduced in \cite{EHO} coincides with the `corresponding vector space dimension', meaning
\[\Dim_-(V_{ip^j-1})\;=\; ip^j\;\in\;\mZ_p. \]
\item We have $\w^{p^{n-1}}\mV\cong\unit$ and $\Dim_-(\mV)=p^{n-1}$ in $\Ver_{p^n}$.
\end{enumerate}

\end{example}

\begin{remark}
We cannot remove the condition $m\le p^{n-1}$ in Theorem~\ref{thm:w}. Indeed, we have $\w^i V_1\not=0$ for all $i$ in $\Ver_3$, since $V_1=F_1(L_1)$ is the odd line. However, some additional cases to Theorem~\ref{thm:w} can be derived from Remark~\ref{rem:staySL2}.
\end{remark}

\begin{remark}\label{rem:MN}
Recall from \cite{BEO} that $V_{p^{n-1}(p-2)}$ is the odd line $\bar{\unit}\in\sVec\subset\Ver_{p^n}$ and that, for $i,j$ as in Example~\ref{ex:w}, we have $V_{(p-2)p^{n-1}+ip^j-1}\cong V_{(p-2)p^{n-1}}\otimes V_{ip^j-1} $, so that
\[\Sym^{ip^j}\left(V_{(p-2)p^{n-1}+ip^j-1}\right)\not=0,\quad \Sym^{ip^j+1}\left(V_{(p-2)p^{n-1}+ip^j-1}\right)=0,\quad\mbox{and}\]
\[\Sym^{p^{n-1}}\left(V_{(p-1)p^{n-1}-1}\right)\not=0,\quad \Sym^{p^{n-1}+1}\left(V_{(p-1)p^{n-1}-1}\right)=0.\]
\end{remark}

%

\subsection{Geometric reductivity}

\begin{theorem}\label{thm:GR}
The category $\Ver_{p^\infty}$ is geometrically reductive.
\end{theorem}

Set $n\in\mZ_{>1}$. We start the proof with the following lemma.

\begin{lemma}\label{lem:diagram2}
For $m:=p^{n-1}$, there is a commutative diagram in $\Ver_{p^n}$
\[\xymatrix{
\PS^m(\mV\otimes \mV)\ar[r]&(\mV\otimes \mV)^{\otimes m}\ar[r]\ar[dr]^{\ev^{\otimes m}}&\Sym^m(\mV\otimes \mV)\ar[r]\ar[d]&0\\
\PE^m(\mV)\otimes \mV^{\otimes m}\oplus \mV^{\otimes m}\otimes \PE^m(\mV)\ar[r]\ar[u]&\mV^{\otimes m}\otimes \mV^{\otimes m}\ar[r]\ar[u]&\unit\ar[r]&0,
}\]
where the rows are exact and the top row is \eqref{eq:PS}.
\end{lemma}
\begin{proof}
Consider the commutative diagram in Lemma~\ref{lem:diagram1} for $V= St_{n-1}$. We can restrict it to $\SL_2$ via the defining $\SL_2\to\GL(St_{n-1})$, obtaining a diagram in $\Tilt \SL_2$. Subsequently, we can apply $F_n$.

Now the diagonal morphism clearly factors through the symmetric power and we obtain the diagram in $\Ver_{p^n}$ in the lemma. The top row is exact by definition, while the bottom row is exact as a consequence of the canonical morphism $\w^{p^{n-1}}\mV\to F_n(\w^{p^{n-1}} V)=\unit$ being an isomorphism, see Theorem~\ref{thm:w}.
\end{proof}

\begin{corollary}\label{cor:split}
For $\ev_\mV:\mV\otimes \mV\to\unit$, the epimorphism $\Sym^{p^{n-1}}(\mV\otimes \mV)\to \unit$ is split.
\end{corollary}
\begin{proof}
The composite diagonal morphism $\mV^{\otimes p^{n-1}}\otimes \mV^{\otimes p^{n-1}}\to\Sym^{p^{n-1}}(\mV\otimes \mV)$ in the diagram of Lemma~\ref{lem:diagram2} factors via $\unit$, providing a splitting.
\end{proof}
\begin{proof}[Proof of Theorem~\ref{thm:GR}]
It suffices to show that $\Ver_{p^n}$ is GR for arbitrarily high $n$.
To show that the finite tensor category $\Ver_{p^n}$ is GR, it suffices to prove that for some epimorphism $P\to\unit$ with $P\in \Ver_{p^n}$ projective, there exists $i>0$ for which $\Sym^{i}P\to\unit$ is split. Since $\mV$, so also $\mV^{\otimes 2}$, is projective, the conclusion follows from Corollary~\ref{cor:split}.
\end{proof}


\section{Geometric reductivity of more general tensor categories}\label{sec:inv}
In this section we let $\k$ be an algebraically closed field of characteristic $p>0$.

\subsection{New connections between conjectures}

The following theorem, reduces two conjectures regarding geometric reductivity to other conjectures in the literature.

\begin{theorem}\label{thm:conj}
Assume that $\Ver_{p^\infty}$ is MN, \it{i.e.} \cite[Conjecture~3.2.3]{CAlg} is valid.
\begin{enumerate}
\item If \cite[Conjecture~1.4]{BEO} is valid for finite tensor categories, then every finite tensor category is GR. Moreover, then every finitely generated ind-algebra in a finite tensor category is of finite type \it{i.e.} \cite[Conjecture~9.3.5]{CEO2} is true.
\item If \cite[Conjecture~1.4]{BEO} is valid, then also \cite[Conjecture~3.1.6(1)]{CAlg} is valid. Concretely, a tensor category of moderate growth $\cC$ is GR if and only if for every finite generated $A\in\CAlg\cC$ the $\bk$-algebra $A^{\inv}$ is finitely generated.
\end{enumerate}
\end{theorem}
The theorem will follow from some unconditional statements, proved in the next subsection.
First we recall the notion, used in Theorem~\ref{thm:conj}, of finite type algebras from \cite{CEO2}:
\begin{definition}
An algebra $A\in\CAlg\cC$ is of {\bf finite type} if the $\bk$-algebra $A^{\inv}:=\Hom(\unit, A)$ is finitely generated and $A$ is a finite $A^{\inv}$-module.
\end{definition}

We recall a classical result concerning finite type algebras. Note that, via Lemma~\ref{lem:315}, it actually implies that finite group schemes are geometrically reductive, see \cite[Theorem~1]{Wa}.
\begin{lemma}\label{lem:Gfin}
Let $G$ be a finite group scheme over $\bk$, and $A$ an ind-algebra in $\Rep G$. Then $A$ is (as a $\bk$-algebra) integral over its subalgebra $A^G$ of invariants. In particular, if $A$ is a finitely generated algebra, then so is $A^G$, and then $A$ is a finite $A^G$-module (\it{i.e.} $A$ is of finite type as an ind-algebra in $\Rep G$). 
\end{lemma}
\begin{proof}
Denote the coaction by $\rho:A\to A\otimes \bk[G]$ and the comultiplication by $\Delta:\bk[G]\to\bk[G]\otimes \bk[G]$. For $a\in A$, denote the characteristic polynomial of multiplication by $\rho(a)$ on the $A$-module $\bk[G]\otimes A$ by $P(t)\in A[t]$. Then we can observe that 
\[\rho(P(t))\;=\;1\otimes P[t]\;\in \bk[G]\otimes A[t],\]
by interpreting the left (resp. right) as the characteristic polynomial of multiplication by $(\bk[G]\otimes\rho)\circ\rho(a)$ (resp. $(\Delta\otimes A)\circ\rho(a)$) on the $\bk[G]\otimes A$-module $\bk[G]\otimes \bk[G]\otimes A$ (for the canonical right action) and observing that the latter two elements are identical. It also follows from Cayley-Hamilton that $P(\rho(a))=0$, which we can thus rewrite as $\rho(P(a))=0$. By injectivity of $\rho$ we find $P(a)=0$, showing integrality.

For the second sentence, let $a_1,\ldots, a_r$ be generators of $A$ as a $\bk$-algebra. Then, for each $i$, there is a (monic) polynomial with coefficients in $A^G$ that annihilates $a_i$. We let $B$ be the subalgebra of $A^G$ generated by the coefficients of these polynomials. Now $A$ is a finite $B$-module by construction, and hence (since $B$ is noetherian) also $A^G$ is a finite $B$-module and thus also finitely generated as a $\bk$-algebra.
\end{proof}

\subsection{Unconditional results}\label{sec:uncon}

In this section, we fix $\cD$ to be a GR + MN tensor category that is also a union $\cup_{n>0}\cD_n$ of finite tensor subcategories $\cD_1\subset\cD_2\subset\cdots$. Note that such $\cD$ must be incompressible, by \cite[Theorem~7.2.3(2)]{CEO2}.

\begin{theorem}\label{thm:comb}
Consider a tensor category $\cC$ with tensor functor $F:\cC\to\cD$. 
\begin{enumerate}
\item If $\cC$ is finite, then it is GR and every finitely generated ind-algebra in $\cC$ is of finite type.
\item If $\cC$ is GR, then for every finitely generated $A\in \CAlg\cC$, the $\bk$-algebra $A^{\inv}$ is finitely generated.
\end{enumerate}
\end{theorem}

Before proving the theorem, we review the notion of Frobenius twists in the sense of \cite{CS}. For $R\in\CAlg\cD$ we let $R^{[1]}$ be the image of the $p$-fold multiplication map $(R^{\otimes p})^{S_p}\to R$ restricted to the $S_p$-invariants $(R^{\otimes p})^{S_p}\subset R^{\otimes p}$. For $s>1$ we define iteratively $R^{[s]}=(R^{[s-1]})^{[1]}$. If~$G$ is an affine group scheme in $\cD$, see \cite{BP, CAlg, CS} for background, then $\bk[G]^{[s]}$ represents a quotient group scheme $G^{[s]}$, leading to a short exact sequence
\[1\to G_{s}\to G\to G^{[s]}\to 1.\]
The $s$-th Frobenius kernel $G_s$ of~$G$ is an infinitesimal group scheme, see \cite[Proposition~4.19]{CS}.

\begin{proof}[Proof of Theorem~\ref{thm:comb}]
For clarity, we will write $\Gamma:\Ind\cD\to\VEC$ for the functor taking invariants in $\cD$ and reserve $(-)^\inv$ for the corresponding functor for $\cC$.
Since any finitely generated algebra in $\CAlg\cC$ belongs to (the ind-completion of) a finitely generated tensor subcategory of $\cC$, we can without loss of generality assume that $\cC$ is finitely generated. It then follows that $F$ takes values in a finitely generated tensor subcategory of $\cD$, hence in one of the finite $\cD_n\subset \cD$. We can and will thus also assume that $\cD$ is a finite tensor category.

Via $F$, and following \cite[\S 8]{Del90}, we can and will realise $\cC$ as a tensor subcategory of the representation category $\Rep G=\Rep_{\cD}G$ in $\cD$ of an affine group scheme $G$ in $\cD$. More concretely the fundamental group $\pi(\cD)$ of $\cD$ has a canonical action on every object of $\cD$, there is a group homomorphism $\varepsilon:\pi(\cD)\to G$, such that the adjoint action of $\pi(\cD)$ on $\bk[G]$ is the canonical action, and $\cC=\Rep(G,\varepsilon)$ is the category of $G$-representations for which the restriction along $\varepsilon$ produces the canonical action. By~\cite[Corollaire~8.18]{Del90}, for an ind-object $X$ in $\cC\subset\Rep G$, we have $X^\inv=X^G$, the maximal $G$-invariant subobject, where the left-hand side is a priori $\Gamma(X^G)$.

Since $\cC$ is finitely generated, $G$ is of finite type, see \cite[Lemma~4.3.3]{CEO2}. By \cite[Proposition~4.19]{CS}, it follows that $G^{[r]}$ is an affine group scheme over $\bk$, for large enough $r\in\mN$, which we now fix.

Let $A$ be a finitely generated ind-algebra in $\cC\subset \Rep G$, which we thus really interpret as an ind-algebra in $\cD$ equipped with an action of $G$. We consider $\bk$-subalgebras
\[\xymatrix{
A^{\inv}=A^G=\Gamma(A^G)\ar@{^{(}->}[r]&\Gamma(A)\ar@{^{(}->}[r]&A\\
A'\ar@{^{(}->}[r]\ar@{^{(}->}[u]&\Gamma(A^{[r]}).\ar@{^{(}->}[u]
}\]
To define $A'$, we must first observe that the co-action $A\to A\otimes \bk[G]$
leads to a co-action
\[\Gamma(A^{[r]})\subset A^{[r]}\to (A\otimes \bk[G])^{[r]}= A^{[r]}\otimes \bk[G^{[r]}].\]
In particular the $G$-action on $A^{[r]}\subset A$ factors via the quotient $G^{[r]}$. Moreover,
since $\Gamma(A^{[r]})$ and $\bk[G^{[r]}]$ both belong to $\VEC\subset\Ind\cD$, it follows that the subalgebra $\Gamma(A^{[r]})$ of $A^{[r]}$ is stable under this action of $G^{[r]}$ (or $G$). Hence, we can define $A'$ as the algebra of $G^{[r]}$-invariants in $\Gamma(A^{[r]})$, which is thus a subalgebra of $A$ on which $G$ acts trivially, justifying the upwards inclusion.

Now, to prove (1), assume that $\cC$ is a finite tensor category. Then $G$ is finite, {\it i.e.} $\bk[G]\in\cD$ as it is a quotient of $\oplus_P F(P\otimes P^\ast$), where $P$ runs over the indecomposable projective objects of $\cC$. We can observe the following:
\begin{enumerate}
\item[(i)] $A^{[r]}$ is finitely generated by \cite[Lemma~4.4(3)]{CS}.
\item[(ii)] $\Gamma(A^{[r]})$ is finitely generated and $A^{[r]}$ is a finite $\Gamma(A^{[r]})$-module by (i) and \cite[Proposition~2.9]{CS}.
\item[(iii)] $A'$ is finitely generated and $\Gamma(A^{[r]})$ is a finite $A'$-module by (ii) and Lemma~\ref{lem:Gfin}.
\item[(iv)] $A$ is a finite $A^{[r]}$-module by \cite[Lemma~4.4]{CS} and hence a finite $\Gamma(A^{[r]})$-module by (ii) and so finally a finite $A'$-module by (iii).
\item[(v)] $A^{\inv}$ is a finite $A'$-module, as a submodule of the finite $A'$-module $A$ from (iv), where $A'$ is noetherian by (iii).
\item[(vi)] $A^{\inv}$ is a finitely generated algebra, by the combination of (iii) and (v). 
\item[(vii)] $A$ is a finite $A^{\inv}$-module, since it is already a finite $A'$-module by (iv).
\end{enumerate}
Claims (vi) and (vii) show that $A$ is of finite type. That $\cC$ is GR then is implied by Lemma~\ref{lem:315}, which concludes the proof of part (1).

Now we prove part (2). We keep notation and assumptions as before the previous paragraph, and assume that $\cC$ is GR. Since $\cD$ is finite, $\pi(\cD)$ is finite. Since $G^{[r]}$ is an ordinary $\bk$-group scheme, the adjoint action of $\pi(\cD)$, via $\pi(\cD)\to G\to G^{[r]}$, on $\bk[G^{[r]}]$ is trivial. In other words, the image of $\pi(\cD)\to G^{[r]}$ is central. We take the quotient $Q$ of $G^{[r]}$ with respect to this central subgroup, leading to a short exact sequence
\[1\to N\to G\to Q\to 1,\]
where $N$ is still a finite group scheme (since $G_r$ and $\pi(\cD)$ were finite). As any tensor subcategory, the tensor subcategory $\Rep_{\bk }Q=\Rep_{\cD}(Q,\varepsilon)$ of $\cC$ is again GR, so $Q$ is geometrically reductive. Now for a finitely generated ind-algebra $A$ in $\cC$, we can observe that
\[A^\inv=A^G=(A^{N})^{Q}.\] 
By part (1) applied to the finite tensor category $\Rep_{\cD}N$, the $\bk$-algebra $\Gamma(A^{N})$ is finitely generated  and $A$ is a finite $\Gamma(A^{N})$-module. Hence, by the Hilbert Basis Property, see \cite[Lemma~6.4.4]{CAlg}, also the submodule $A^{N}\subset A$ is a finite $\Gamma(A^{N})$-module, so $A^{N}$ is a finitely generated algebra. Now by the classical version of the theorem (Nagata's Theorem from \cite{Na}) applied to the geometrically reductive group $Q$, it follows that $A^{\inv}$ is finitely generated. 
\end{proof}

\begin{remark}
\begin{enumerate}
\item Theorem~\ref{thm:comb}(1) for the special case $\cD=\Ver_p$ already appeared as \cite[Theorem~1.2]{Ve}. Many ideas in the proof of the latter appear again in the above proof.
\item Another approach to proving Theorem~\ref{thm:comb}(2) would by to mimic the classical proof of \cite{Na, BF} closer, rather than reducing to it.
\end{enumerate}

\end{remark}

\subsection{Conclusion}
By construction $\Ver_{p^\infty}=\cup_n\Ver_{p^n}$ is a union of finite tensor categories. It is GR by Theorem~A. If we thus assume the conjecture that it is also MN, it satisfies the conditions on $\cD$ in \S \ref{sec:uncon}. Further, \cite[Conjecture~1.4]{BEO} predicts that any tensor of moderate growth admits a tensor functor to $\Ver_{p^\infty}$. The latter two conjectures thus predict that for any appropriate tensor category, the assumptions in Theorem~\ref{thm:comb} are satisfied, so that Theorem~\ref{thm:conj} follows.

\subsection*{Acknowledgements}
The author thanks Pavel Etingof, Joseph Newton and Alexander Sherman for useful discussions. The research was partly supported by ARC grants FT220100125 and DP250100762.


\begin{thebibliography}
	{EGNO}
	
	

\bibitem[BNPS]{Stein} C.P.~Bendel, D.K.~Nakano, C.~Pillen, P.~Sobaje:
On tensoring with the Steinberg representation.
Transform. Groups 25 (2020), no. 4, 981--1008.

\bibitem[BE]{BE}  D.~Benson, P.~Etingof: Symmetric tensor categories in characteristic 2. Adv. Math. 351 (2019), 967--999.

\bibitem[BEO]{BEO} D.~Benson, P.~Etingof, V.~Ostrik: New incompressible symmetric tensor categories in positive characteristic. Duke Math. J. 172 (2023), no. 1, 105–200.

\bibitem[BP]{BP} D.~Benson, J.~Pevtsova: Group schemes and their Lie algebras over a symmetric tensor category. arXiv:2507.02031.





\bibitem[BF]{BF}H.~Borsari, W.~Ferrer Santos:
Geometrically reductive Hopf algebras.
J. Algebra 152 (1992), no. 1, 65--77. 



\bibitem[Co1]{Selecta} K.~Coulembier: Tensor ideals, Deligne categories and invariant theory. Selecta Math. (N.S.) 24 (2018), no. 5, 4659--4710. 



\bibitem[Co2]{AbEnv}  K.~Coulembier: Monoidal abelian envelopes. Compos. Math. 157 (2021), no. 7, 1584--1609.

\bibitem[Co3]{CAlg}  K.~Coulembier: Commutative algebra in tensor categories. Transform. Groups 31 (2026), no. 2, 1225--1272.

\bibitem[Co4]{Geom} K.~Coulembier: Algebraic geometry in tensor categories. To appear in Michigan Mathematical Journal. arXiv:2311.02264


\bibitem[CEO1]{CEO} K.~Coulembier, P.~Etingof, V.~Ostrik: On Frobenius exact symmetric tensor categories. With an appendix by A.~Kleshchev. Ann. of Math. (2) 197 (2023), no. 3, 1235--1279.


\bibitem[CEO2]{CEO2}	K.~Coulembier, P.~Etingof, V.~Ostrik: Incompressible tensor categories. Adv. Math. 457 (2024), Paper No. 109935, 65 pp.
	 


\bibitem[CS1]{CS} K.~Coulembier, A.~Sherman:
Homogeneous spaces in tensor categories. arXiv:2505.04848.

\bibitem[CS2]{CS2} K.~Coulembier, A.~Sherman:
Absolutely flat algebras in tensor categories. In preparation.

\bibitem[Dec]{Dec25} T.D.~Décoppet: Higher Verlinde Categories: The Mixed Case. arXiv:2407.20211.


	
	\bibitem[De1]{Del90} P.~Deligne: Cat\'egories tannakiennes. The Grothendieck Festschrift, Vol. II, 111--195, Progr. Math., 87, Birkh\"auser Boston, Boston, MA, 1990. 
	
	\bibitem[De2]{Del02} P.~Deligne: Cat\'egories tensorielles. Mosc. Math. J. 2 (2002), no. 2, 227--248.	 
	

\bibitem[EHO]{EHO}P.~Etingof, N.~Harman, V.~Ostrik:
$p$-adic dimensions in symmetric tensor categories in characteristic $p$.
Quantum Topol. 9 (2018), no. 1, 119–140.
		
Lecture Notes in Mathematics, 900. Springer-Verlag, Berlin-New York, 1982, pp. 101-228.

	 

\bibitem[EG]{EG} P.~Etingof, S.~Gelaki: Finite symmetric tensor categories with the Chevalley property in characteristic 2. J. Algebra Appl. 20 (2021), no. 1, Paper No. 2140010.

\bibitem[EGNO]{EGNO}P.~Etingof, S.~Gelaki, D.~Nikshych, V.~Ostrik:
Tensor categories. 
Mathematical Surveys and Monographs, 205. American Mathematical Society, Providence, RI, 2015. 



\bibitem[Ha]{Ha} W.J.~Haboush:
Reductive groups are geometrically reductive.
Ann. of Math. (2) 102 (1975), no. 1, 67--83.


\bibitem[Ja]{Jantzen}
J.C.~Jantzen:
Representations of algebraic groups. 
 Mathematical Surveys and Monographs, 107. American Mathematical Society, Providence, RI, 2003.






\bibitem[Mu]{Mu}
D.~Mumford:
Geometric invariant theory.
Ergebnisse der Mathematik und ihrer Grenzgebiete, (N.F.), Band 34
Springer-Verlag, Berlin-New York, 1965.


\bibitem[Na]{Na} M.~Nagata:
Lectures on the fourteenth problem of Hilbert.
Tata Institute of Fundamental Research, Bombay, 1965.


\bibitem[Ne]{Ne}
J.~Newton: Higher Verlinde categories of reductive groups. arXiv:2601.11084









\bibitem[Ve]{Ve} S.~Venkatesh:
Hilbert basis theorem and finite generation of invariants in symmetric tensor categories in positive characteristic.
Int. Math. Res. Not. IMRN 2016, no. 16, 5106--5133.
 
\bibitem[Wa]{Wa} W.C.~Waterhouse:
Geometrically reductive affine group schemes.
Arch. Math. (Basel) 62 (1994), no. 4, 306--307. 

 	\end{thebibliography}
\end{document}